\documentclass[12pt]{article}

\usepackage[margin=1in]{geometry}
\usepackage{amsmath,amsfonts,graphicx,framed}
\usepackage{amssymb}
\usepackage{amsthm}
\usepackage{fancyhdr}
\usepackage{float}
\usepackage{caption}
\usepackage{comment}
\usepackage[utf8]{inputenc}
\usepackage{authblk}
\usepackage[mathscr]{euscript}
\usepackage{mathtools}
\usepackage{xcolor}

\usepackage[hypertexnames=false]{hyperref}

\theoremstyle{plain}
\newtheorem{thm}{Theorem}
\newtheorem{lemma}{Lemma}

\newtheorem{prop}{Proposition}
\theoremstyle{definition}

\newtheorem{rem}{Remark}

\numberwithin{equation}{section}
\numberwithin{thm}{section}
\numberwithin{prop}{section}
\numberwithin{lemma}{section}
\numberwithin{cor}{section}
\numberwithin{defn}{section}
\numberwithin{rem}{section}
\numberwithin{ex}{section}

\newcommand{\pd}{\partial}

\newcommand{\mbb}{\mathbb}

\newcommand{\ep}{\varepsilon}
\newcommand{\vp}{\varphi}
\newcommand{\re}{\mbb R}
\newcommand{\al}{\alpha}

\newcommand{\Om}{\Omega}
\newcommand{\eqal}[1]{\begin{equation}\begin{aligned}#1\end{aligned}\end{equation}}

\newcommand{\ov}{\overline}

\newcommand{\msk}{\medskip}

\begin{document}

\title{Hessian estimates for special Lagrangian equation by doubling}

\author{Ravi Shankar}

\date{\today}

\maketitle

\begin{abstract}
    New, doubling proofs are given for the interior Hessian estimates of the special Lagrangian equation.  These estimates were originally shown by Chen-Warren-Yuan in CPAM 2009 and Wang-Yuan in AJM 2014.  This yields a higher codimension analogue of Korevaar's 1987 pointwise proof of the gradient estimate for minimal hypersurfaces, without using the Michael-Simon mean value inequality.
\end{abstract}

\section{Introduction}

\textbf{The pointwise estimate for minimal surfaces}

In 1987, Korevaar \cite{K87} gave a new, pointwise proof of the gradient estimate for solutions of the minimal hypersurface PDE.  The proof was modelled after Cheng-Yau's cutoff \cite{CY76} in the maximal surface context.  Korevaar's pointwise proof was robust enough to give gradient estimates for fully nonlinear relatives, the sigma-k curvature equations. 

\smallskip
The original proof by Bombieri-De Giorgi-Miranda \cite{BDM69}, and simplified by Trudinger \cite{T72}, uses two tools from minimal surface theory: the Michael-Simon mean value inequality for graphs with bounded mean curvature, and the Jacobi inequality $\Delta b\ge |\nabla b|^2$, a strong subharmonicity originating from Jacobi fields in the vertical direction.  The Korevaar proof relies only on the Jacobi inequality.  The two dimensional surface proof of Gregori \cite{G94} uses isothermal coordinates.

\medskip
Although the Jacobi inequality can sometimes be found in other categories using ordinary differential calculus, the mean value inequality, and its cousin the monotonicity formula, is a delicate integral relation which is difficult to establish outside the minimal surface context.  

\medskip
\noindent
\textbf{Higher codimensions?}

Despite its versatility, an analogous Korevaar argument is missing for higher codimension minimal surfaces.  Wang \cite{W04} established a gradient estimate under the area decreasing condition \cite{W04} using an integral method.  More recently, Dimler \cite{D23} found a pointwise proof under the area decreasing condition, using Savin's theory of viscosity solutions.  However, the method requires an additional condition, that all but one component of the graph to be small.

\medskip
\noindent
\textbf{Main result of this paper}

In this paper, we find a Korevaar type proof of the gradient estimate for a class of high codimension minimal surfaces.  These surfaces can be described by a single potential function $u$, such that $(x,Du)$ is a minimal surface.  The potential solves a second order, fully nonlinear, elliptic PDE \eqref{slag} called the special Lagrangian equation, as shown by Harvey-Lawson \cite{HL82}.  In this context, the gradient estimate for $(x,Du)$ is a Hessian estimate for $u$.

\medskip
Despite its relative simplicity compared to general high codimension surfaces, a Korevaar proof of the Hessian estimate for the special Lagrangian equation was elusive.  Integral proofs under much weaker, and sharp by \cite{NV10}-\cite{WY13}-\cite{MS} singular solutions, conditions were only established by Chen-Warren-Yuan \cite{CWY09} in 2009, and Wang-Yuan in 2014 \cite{WY14}.  A pointwise proof was attempted in \cite{WY08} but required a flatness condition on the gradient.  

\medskip
The main technical ingredient of our proof is a Korevaar type pointwise calculation.  The other ingredients are also pure PDE techniques, described below.  In particular, nowhere is the Michael-Simon mean value inequality used.

\medskip
\noindent
\textbf{A doubling approach.}

Our approach to the Hessian estimate is based on Shankar-Yuan's resolution of the Hessian estimate for the sigma-2 equation in dimension four \cite{SY23a}.  The first step is to derive partial regularity by combining an Alexandrov ($D^2u$-existing-a.e.) theorem with Savin's $\ep$-regularity \cite{S07}: the singular set is closed and Lebesgue measure zero.  The next step is to propagate this partial regularity to the entire domain using a doubling inequality for the Hessian.  Partial regularity implies local boundedness of the Hessian inside the smooth set, so the doubling gives a global $C^{1,1}$ estimate and rules out the singular set.

\medskip
The doubling inequality generally requires a Jacobi field type inequality $\Delta b\ge |\nabla b|^2$.  Trudinger \cite{T80} showed doubling in the uniformly elliptic Harnack inequality context.  Using the Guan-Qiu test function \cite{GQ19}, Qiu \cite{Q17} established doubling for the sigma-2 equation in dimension three, for which Jacobi is available.  In fact, the sigma-2 equation is a special Lagrangian equation, in dimension three only.  Shankar-Yuan showed doubling for the sigma-2 equation in dimension four using an almost-Jacobi inequality with a degenerate coefficient \cite{SY23a}.  Shankar-Yuan found a geometric doubling inequality for the Monge-Amp\`ere equation \cite{SY23b}.

\medskip
In the present paper, we use the Jacobi inequalities of Chen-Warren-Yuan \cite{CWY09} and Wang-Yuan \cite{WY14} to discover doubling inequalities for the special Lagrangian equation in convex and critical/supercritical phase categories.

\medskip
Another partial regularity propagation has been used for the minimal hypersurface equation.  Caffarelli-Wang \cite{CW93} has another proof of the $C^{1,\alpha}$ regularity of Lipschitz solutions.  Starting with $C^{1,\alpha}$ partial regularity (pg 155), they use a geometric Harnack inequality to propagate this flatness to the entire domain (pg 156).

\medskip
\noindent
\textbf{New ideas to establish the doubling inequality}

We modify the Korevaar type calculation to our high codimension setting.  This fails to give a Hessian estimate, but it yields a doubling inequality.  Two modifications are needed to Korevaar to achieve this.  We first mix Guan-Qiu's test function involving the radial derivative $x\cdot Du-u$ with the Korevaar cutoff to create a minimal surface version of Guan-Qiu.  Secondly, for critical phases, the equation's ellipticity and concavity degenerate, and we need to add an additional increasing, concave term to the cutoff to compensate.  Unfortunately, only Green's type functions have strong enough concavity, and the cutoff becomes singular.  Nevertheless, we only need to establish a doubling inequality, rather than a Hessian estimate.  We are free to exclude a small sphere from our calculations.  We can then place the singularity inside this inner sphere without analytic problems.

\medskip
The Qiu \cite{Q17} cutoff used for sigma-2 in three dimensions (i.e. critical phase sLag in 3D) does not seem to extend to the special Lagrangian equation in the convex or higher dimensional critical phase settings.  Our singular cutoff of Korevaar/Guan-Qiu type seems important to obtain the doubling inequality.

\medskip
We note that Wang-Yuan \cite{WY14} established n-1 convexity of solutions.  This slightly weaker version of convexity and the black box in Chaudhuri-Trudinger \cite{CT05} allow us to establish Alexandrov regularity without any trouble.  In other situations, Alexandrov regularity can be challenging.


\section{Statement of results}

\smallskip
This paper gives pointwise proofs of the Hessian estimates for the special Lagrangian equation:
\eqal{
\label{slag}
\sum_{i=1}^n\arctan\lambda_i(D^2u)=\Theta=\text{ constant}\in \Big(-n\frac{\pi}{2},n\frac{\pi}{2}\Big)
}
Here, $\lambda_i$'s are the eigenvalues of the Hessian $D^2u$ of solution $u(x)$.  The symmetric polynomial $\sigma_k$ version of this equation is
$$
\cos\Theta\,(\sigma_1-\sigma_3+\sigma_5-\cdots)-\sin\Theta\,(1-\sigma_2+\sigma_4-\cdots)=0.
$$
Harvey and Lawson \cite{HL82} showed that Lagrangian graph $(x,Du(x))\in(\re^n\times\re^n,dx^2+dy^2)$ is a volume minimizing submanifold.  The phase is called critical or supercritical if $\Theta\ge (n-2)\pi/2$ \cite{Y06}.  In this case, Yuan showed that the PDE has convex level set.

\medskip
The result of this paper is a new proof of the following two Hessian estimates.  The first was shown by Chen-Warren-Yuan \cite{CWY09} in 2009.  Their estimate was explicit, while our proof is by compactness.
\begin{thm}[Convex solutions]
Let $u$ be a smooth convex solution of \eqref{slag} in $B_2(0)$.   Then 
$$
|D^2u(0)|\le C\Big(n,\|u\|_{C^{0,1}(B_1(0))},\Theta\Big).
$$
\end{thm}

Stronger forms of the next estimate were shown by Warren-Yuan \cite{WY09a}, \cite{WY10} and Wang-Yuan \cite{WY14} for $n\ge 3$ and Warren-Yuan \cite{WY09b} in dimension two.  We also restrict to $n\ge 3$.  The dimension two case is either harmonic, or covered by the simple compactness method in \cite{L19}.  These two dimensional cases were first consequences of results by Heinz \cite{H59} and Gregori \cite{G94} using isothermal coordinates.
\begin{thm}[Critical phase]
Let $u$ be a smooth solution of \eqref{slag} on $B_2(0)$ for phase $\Theta$ critical $\Theta=(n-2)\pi/2$ or supercritical $\Theta\in ((n-2)\pi/2,n\pi/2)$ for $n\ge 3$.  Then 
$$
|D^2u(0)|\le C\Big(n,\|u\|_{C^{0,1}(B_1(0))},\Theta\Big).
$$
\end{thm}

A byproduct is a removal of the flatness condition in the pointwise proof of \cite{WY08}, and a generalization of their condition required for the estimate.  Given a smooth solution $u$ of \eqref{slag}, we say that positive, proper, smooth function $a(D^2u)$ of the Hessian has a \textit{Jacobi inequality} if $\Delta_ga\ge 2|\nabla_g a|^2/a$.   We also recall that a semi-convex function has a Hessian lower bound $D^2u\ge -KI$ for some $K>0$, and that a proper function $a$ satisfies $a^{-1}(B)$ is bounded for any bounded set $B$.
\begin{thm}[Semiconvex + Jacobi]
\label{thm:scvx}
Let $u$ be a smooth solution of \eqref{slag} on $B_2(0)$ which is semi-convex and has a Jacobi inequality.  Then
$$
|D^2u(0)|\le C\Big(n,\|u\|_{C^{0,1}(B_1(0))},a,K,\Theta\Big).
$$
\end{thm}

\begin{rem}
    One consequence is a new proof of the following.  The Hessian estimate was earlier shown in a pointwise proof of \cite[Lemma 2.2]{WY08} assuming the Hessian eigenvalue condition
    \eqal{
    \label{cD19}
    3+(1-\ep)\lambda_i^2+2\lambda_i\lambda_j\ge 0,\qquad 1\le i,j\le n,
    }
    for some $\ep>0$, under an additional flatness condition $|Du(x)|\le \delta(n)|x|$.  Later, this estimate was shown in \cite[Theorem 5.1]{D19} without flatness, for a slightly negative $\ep$ above, using the Michael-Simon mean value inequality.  Theorem \ref{thm:scvx} shows how to remove the Warren-Yuan flatness condition on $Du$ in the pointwise proof. Indeed, equation (4.52) and Lemma 4.1 in \cite{D19} show that $u$ is semi-convex under condition \eqref{cD19}.  The Jacobi inequality in \cite[Lemma 2.1]{WY08} then verifies that the assumptions for Theorem \ref{thm:scvx} are verified.  It is likely that doubling proofs for the sub-critical estimates in \cite{Z22,Z23} are also possible.
\end{rem}

\begin{rem}
    In view of Mooney-Savin's \cite{MS} recent semi-convex singular solution of \eqref{slag}, it is reasonable to expect that the Jacobi inequality required for Theorem \ref{thm:scvx} fails for such solutions.  In fact, after a Legendre-Lewy transform, their solution satisfies $\det D^2\bar u=0$ on a subdomain.  This equation lacks ellipticity and concavity, which seems necessary to establish Jacobi inequalities.
\end{rem}

\section{Preliminaries}

\textbf{1. Notation:} for function $u(x)$, we denote $u_i=\pd u/\pd x^i$ and $u_{ij}=\pd^2u/\pd x^i\pd x^j$.  On the other hand, eigenvalues $\lambda_i$ of the Hessian and subharmonic quantities $b_m=m^{-1}(\ln\sqrt{1+\lambda_1^2}+\cdots+\ln\sqrt{1+\lambda_m^2})$ do not denote partial derivatives.  Moreover, $C(n)$ denotes various dimensional constants.

\noindent
\textbf{2. Differential Operators:} For $g=dx^2+dy^2|_{y=Du}$, or $g=I+D^2uD^2u$, the Laplace-Beltrami operator is
\eqal{
\Delta_g=\frac{1}{\sqrt{\det g}}\pd_i\Big(\sqrt{\det g}\,g^{ij}\pd_j\Big).
}
In fact, the mean curvature is $H=\Delta_g(x,Du(x))$.  By Harvey-Lawson \cite{HL82}, it follows that $H=0$ on solutions of \eqref{slag}.  Since this implies $x^i$ are harmonic coordinates where $\Delta_gx=0$, the Laplace operator simplifies to the linearized operator of \eqref{slag}:
\eqal{
\Delta_g=g^{ij}\pd_{ij}\stackrel{p}{=}\frac{1}{1+\lambda_i^2}\pd_{ii}
}
at any diagonal point $\lambda_i=u_{ii}$ of the Hessian, such as after composition with a rotation.  Here, we assume summation over repeated indices, unless the index ranges are stated.  The gradient and inner product with respect to the metric are
\eqal{
&\nabla_gv=\Big(g^{1i}v_i,\dots,g^{n1}v_i\Big),\\
&\langle \nabla_gv,\nabla_gw\rangle_g=g^{ij}v_iw_j\stackrel{p}{=}g^{ii}v_iw_i,\\
&|\nabla_gv|^2=\langle \nabla_gv,\nabla_gv\rangle_g\stackrel{p}{=}g^{ii}v_i^2.
}
where $v_i=\pd_iv$ for function $v$.  

\smallskip
\noindent
\textbf{3. Jacobi inequality: convex.}  We recall the Jacobi inequality for smooth convex solutions of \eqref{slag}.  Given volume form $dV_g=\sqrt{\det g}dx$, we define
\eqal{
V=\sqrt{\det g}=\sqrt{\det(I+(D^2u)^2)}=\Pi_{i=1}^n\sqrt{1+\lambda_i^2}.
}
Then the Jacobi inequality is established by directly taking derivatives and using algebra.
\begin{prop}[\cite{CWY09} Proposition 2.1]
Let $u$ be a smooth \textbf{convex} solution of \eqref{slag} on $B_R(0)\subset\re^n$.  Then
\eqal{
\Delta_g\ln V\ge \frac{1}{n}|\nabla_g \ln V|^2,
}
or equivalently, for $a=V^{1/n}$,
\eqal{
\label{JacCon}
\Delta_ga\ge 2\frac{|\nabla_ga|^2}{a}.
}
\end{prop}

\noindent
\textbf{4. Jacobi inequality: critical or supercritical phase.}  We order the eigenvalues of the Hessian by $\lambda_1\ge\cdots\ge \lambda_n$.

\begin{prop}[\cite{WY14} Lemma 2.3]
\label{prop:JacCrit}
    Let $u$ be a smooth solution of \eqref{slag} with $\Theta\ge(n-2)\pi/2$.  Suppose $u$ is smooth near $x=p$ and that at $x=p$, $\lambda_1=\cdots=\lambda_m>\lambda_{m+1}$.  Then the function $b_m=m^{-1}\sum_1^m\ln\sqrt{1+\lambda_m^2}$ is smooth near $x=p$, and satisfies
    \eqal{
    \Delta_gb_m\ge M|\nabla_gb_m|^2, \qquad M:=\Big(1-\frac{4}{\sqrt{4n+1}+1}\Big),
    }
    or equivalently for $a_m=\exp(M b_m)$:
    \eqal{
    \label{JacCrit}
    \Delta_ga_m\ge 2\frac{|\nabla_ga_m|^2}{a_m}.
    }
\end{prop}

\noindent
Note that $b_m$ is symmetric in the degenerate eigenvalues.  See \cite[Theorem 5.1]{A07} for the second derivative calculation of symmetric eigenvalue functions.  One can take a degenerate limit in this calculation if $b_m$ is symmetric.

\noindent
\textbf{5. The $n-1$ convexity for critical or supercritical phases:} We recall the following eigenvalue pinching obtained from \eqref{slag} for large phases using trigonometric identities.

\begin{lemma}[\cite{WY14} Lemma 2.1]
    Suppose the ordered numbers $\lambda_1\ge\cdots\ge\lambda_n$ solve \eqref{slag} with $\Theta\ge (n-2)\pi/2$ and $n\ge 2$.  Then
    \begin{align}
    \label{order}
        \lambda_1\ge\lambda_2\ge&\cdots\ge\lambda_{n-1}>0\text{ and }\lambda_{n-1}\ge|\lambda_n|,\\
        &\lambda_1+n\lambda_n\ge 0,\\
        \label{n-1}
        \sigma_k(\lambda_1,\dots,&\lambda_n)\ge 0\text{ for all }1\le k\le n-1.
    \end{align}
\end{lemma}

\noindent
Condition \eqref{n-1} is called \textbf{\textit{n-1} convexity}.  More generally, we say $u$ is $k$-convex if $\sigma_\ell\ge 0$ for $1\le \ell\le k$.  It is interpreted in the viscosity sense for non-smooth $u$.  Condition \eqref{order} is related to the convexity of the PDE level set $F^{-1}\{0\}$, since derivative $1/(1+\lambda_i^2)$ is increasing with $i$; see \cite{Y06} for a proof of this fact.

\smallskip
\noindent
\textbf{6. Closedness of viscosity subsolutions.  } We say that $u$ is a viscosity subsolution of a fully nonlinear elliptic PDE $F(D^2u)=0$, i.e. locally uniformly continuous $F$ satisfies $F(M+N)>F(M)$ for any $N>0$ at each matrix $M$ in a convex cone of symmetric matrices containing the positive definite ones, if $F(D^2Q)\ge 0$ for each quadratic $Q$ touching $u$ from above near a point, or $Q(x_0)=u(x_0)$ with $Q\ge u$ near $x_0\in\Omega$; see \cite[Proposition 2.4]{CC95}.  A smooth viscosity subsolution satisfies $F(D^2u)\ge 0$ pointwise.  A supersolution satisfies the reverse inequality, and a solution is both a subsolution and a supersolution.

\smallskip
Special Lagrangian equation \eqref{slag} is elliptic, and $\sigma_{k}$ is elliptic on the cone of $k$-convex matrices, or $M$'s satisfying $\sigma_\ell(M)\ge 0$ for $1\le \ell\le k$; see \cite[eq. (2.3)]{TW99} for this and similar basics of $k$-convexity.  

\smallskip
We will use the standard fact that the uniform limit of a sequence of viscosity solutions of an elliptic equation is also viscosity.  This is stated in \cite{CC95}, and this basic proof is written down in, for example, \cite[Appendix]{SY23a}.  Since the domain of the special Lagrangian equation is entire, this is clear.  For $k$-convexity, we repeat the proof verbatim to show the known fact that a uniform limit of $k$-convex solutions is also $k$-convex.

\begin{lemma}[\cite{CC95}]
\label{lem:n-1}
    Let $u_k\in C(\Om)$ be a sequence of $k$-convex functions converging uniformly to $u\in C(\Om)$.  Then $u$ is $k$-convex.
\end{lemma}

\noindent
\textbf{7. Alexandrov theorem on bounded domains.}  The condition of $k$-convexity leads to an Alexandrov theorem, which is standard for convex, i.e. $n$-convex functions.  Let us verify the standard fact that the ``black box" still works on bounded domains.

\begin{prop}[\cite{CT05}, Theorem 1.1]
\label{prop:Alex}
    Let $u\in C(\Omega)$ for a domain $\Omega\subset\re^n$ with $n\ge 2$.  Suppose $u$ is k-convex for $k>n/2$.  Then $u$ is twice differentiable almost everywhere in $\Omega$.  More precisely, for almost every $x_0\in\Om$, there is a quadratic $Q$ such that $u(x)-Q(x)=o(|x-x_0|^2)$.
\end{prop}
\begin{proof}
    Theorem 1.1 in \cite{CT05} works if $\Omega=\re^n$, so it suffices to extend $u$ to a $k$-convex function $\re^n$ outside a small neighborhood of any point $x_0\in\re^n$.  Since $u$ is continuous, we can choose a tall enough convex polynomial $P$ such that $P(x_0)<0$ but $P(x)>u(x)$ on some $\pd B_r(x_0)\subset\subset \Omega$.   Since $P$ is convex, it is $k$-convex, so if we define $\bar u:=\max(P,u)$ on $B_r(x_0)$ and $\bar u=P$ outside $B_r(x_0)$, this is a  viscosity subsolution of $\sigma_\ell\ge 0$, hence $k$-convex.  By Alexandrov theorem \cite[Theorem 1.1]{CT05}, we conclude $\bar u$ is second order differentiable almost everywhere, hence $u$ is also Alexandrov on a neighborhood of $x_0$.  Varying $x_0\in\Om$, we conclude the proof.
\end{proof}

\noindent
\textbf{8. Savin small perturbation theorem.  }We restate Savin \cite[Theorem 1.3]{S07} for equations $F(M)$ only depending on the Hessian, defined on $\text{Sym}(n;\re)$

\begin{prop}[\cite{S07}, Theorem 1.3]
\label{prop:Savin}
Let $F(D^2u)$ satisfy the following hypotheses:

(i) $F$ is elliptic, $F(M+N)>F(M)$ if $N>0$,

(ii)  $F(0)=0$,

(iii) $F\in C^2$.

\noindent
Then there exists $c_1$ small enough depending on $n,F$ such that if viscosity solution $u$ of $F(D^2u)=0$ satisfies flatness $\|u\|_{L^\infty(B_1(0))}\le c_1$, then $u\in C^{2,\alpha}(B_{1/2})$ with $\|u\|_{C^{2,\al}(B_{1/2})}\le 1$.  
\end{prop}

\noindent
\textbf{9. Partial regularity.  }Suppose $u$ is a viscosity solution of \eqref{slag} with Alexandrov, or $D^2u$ exists a.e..  Then by combining with Savin, we deduce the singular set of $u$ is closed and measure zero (partial regularity).  Indeed, if $u-Q=o(|x|^2)$, then one can apply Savin Proposition \ref{prop:Savin} to $v_r(x)=(u(rx)-Q(rx))/r^2=o(r^2)/r^2$.  This function is flat and solves $G(D^2v)=F(D^2Q+D^2v)=0$ with $G(0)=0$, so Savin gives $C^{2,\al}$ regularity nearby the Alexandrov point.  This shows the set of second order differentiable points is open, full measure, and contained in the $C^\infty$ set.   In particular, we have partial regularity for \eqref{slag} in the cases of convex solutions and critical or supercritical phases.

\section{Doubling for convex or semi-convex solutions}

This section establishes a doubling inequality for the Hessian.  First the convex solution case, with Jacobi inequality \eqref{JacCon}.  Recall that a proper function $f$ satisfies $f^{-1}(B)$ bounded for bounded set $B$.

\begin{prop}
\label{prop:doubCon}
    Let $u\in C^\infty$ solve \eqref{slag} on $B_2(0)$ with $u$ convex.  Then for any $y\in B_{1/2}(0)$, there exists $R(n,\|u\|_{C^{0,1}(B_1(0))})>0$ small enough such that for any $r\le R$,
    \eqal{
    \sup_{B_R(y)}a(D^2u)\le C(r,n,\|u\|_{C^{0,1}(B_1(0))})\sup_{B_r(y)}a(D^2u).
    }
    By the properness of $a=V^{1/n}$, we obtain
    \eqal{
    \sup_{B_R(y)}|D^2u|\le C\Big(r,n,\|u\|_{C^{0,1}(B_1(0))},\sup_{B_r(y)}|D^2u|\Big).
    }
\end{prop}

Next the semi-convex solution with a Jacobi inequality case.

\begin{prop}
\label{prop:doubSemi}
    Let $u\in C^\infty$ solve \eqref{slag} on $B_2(0)$ with $u$ semi-convex $D^2u\ge -KI$ with a Jacobi inequality $\Delta_ga\ge 2|\nabla_ga|^2/a$.  Then for any $y\in B_{1/2}(0)$ and $0<r\le 1/4$, there exists $R(n,K,\|u\|_{C^{0,1}(B_1(0))})>0$ small enough such that
    \eqal{
    \sup_{B_R(y)}a(D^2u)\le C(r,K,n,\|u\|_{C^{0,1}(B_1(0))})\sup_{B_r(y)}a(D^2u).
    }
    By the properness of $a$, we obtain
    \eqal{
    \sup_{B_R(y)}|D^2u|\le C\Big(r,K,n,\|u\|_{C^{0,1}(B_1(0))},a,\sup_{B_r(y)}|D^2u|\Big).
    }
\end{prop}

We prove this in two separate cases of the same calculation.

\begin{proof}

\medskip
Letting $h\ll 1$ and $t\gg 1$ with $y\in B_{1/2}(0)$, we define a Korevaar \cite{K87} exponential cutoff using a Guan-Qiu \cite{GQ19} type radial derivative for the phase:
\eqal{
\eta=\Big(e^{(1-\vp)/h}-1\Big)_+,\qquad \vp=(x-y)\cdot Du-u+u(y)+t|x-y|^2/2.
}
We make sure $t\ge C(\|u\|_{C^{0,1}(B_1)})$ is large enough for $\vp> 1$ on $\pd B_{1/2}(y)$.  Then $\text{supp}(\eta)\subset \subset B_{1}(0)$.  Also $B_r(y)\subset\subset \text{supp}(\eta)$ for $r\le R(t,\|u\|_{C^{0,1}(B_1)})$ small enough.  Note that we continuously extend $\eta=0$ outside the connected component of $B_1(0)$ containing $x=y$.

\smallskip
We now start with a standard calculation.  At the max point of $\eta a$, we know
\eqal{
\nabla_g a=-a\nabla_g\eta/\eta,
}
so the Jacobi implies
\eqal{
\label{reverse_Jac}
0&\ge a\Delta_g \eta+2\langle\nabla_g\eta,\nabla_g a\rangle+\Delta_g a\\
&=a\Delta_g\eta+\eta(\Delta_g a-2|\nabla_g a|^2/a)\\
&\ge a\Delta_g\eta.
}
Therefore,
\eqal{
\label{Jac}
|\nabla_g\vp|^2\le h\Delta_g\vp.
}
The RHS at a diagonal point $u_{ii}=\lambda_i$ with $\lambda_1\ge\lambda_n$ (omitting sums):
\eqal{
\Delta_g\vp=\frac{\lambda_i+t}{1+\lambda_i^2}\le Ct.
}
The LHS:
\eqal{
|\nabla_g\vp|^2=(x_i-y_i)^2\frac{(\lambda_i+t)^2}{1+\lambda_i^2}.
}
\textbf{(i) Suppose $u$ is convex:} since $\lambda_i\ge 0,t>1$, we get
\eqal{
(x_i-y_i)^2\frac{t^2+2t\lambda_i+\lambda_i^2}{1+\lambda_i^2}\ge |x-y|^2.
}
We obtain
\eqal{
|x-y|^2\le Cht\le r^2
}
if $h=r^2/Ct$.  

\smallskip
Therefore, the maximum value occurs on the boundary of $B_1(0)\setminus B_r(y)$.  Since $\eta=0$ on $\pd B_1(0)$, it occurs on $\pd B_r(y)$.  Using $\eta>0$ on $B_R(y)$, we obtain the doubling inequality
\eqal{
\sup_{B_R(y)}a&\le C\sup_{B_R(y)}\eta a\\
&\le C\sup_{B_1(0)}\eta a\\
&\le C\sup_{\pd B_r(y)}\eta a\\
&\le C\sup_{B_r(y)}a.
}
Here, $C=C(r,n,\|u\|_{C^{0,1}(B_1(0)})$.  

\medskip
\noindent
\textbf{(ii) Suppose $u$ semi-convex $\lambda_i\ge -K$.}  We first ensure $t\ge 2K$.  Suppose $|x_i-y_i|\ge |x-y|/\sqrt{n}$ for some $1\le i\le n$.

\smallskip
\noindent
\textbf{Subcase $\lambda_i\le 3K$.}  Then by 
$$
(\lambda_i+t)^2=\Big((\lambda_i+K)+(t-K)\Big)^2\ge (t-K)^2\ge K^2,
$$
we get
\eqal{
(x_i-y_i)^2\frac{(\lambda_i+t)^2}
{1+\lambda_i^2}\ge (x_i-y_i)^2\frac{K^2}{1+9K^2}\ge c|x-y|^2.
}
\textbf{Subcase $\lambda_i>3K$.}  Supposing also $t>1$, then as in the convex case,
\eqal{
(x_i-y_i)^2\frac{t^2+2t\lambda_i+\lambda_i^2}{1+\lambda_i^2}\ge |x-y|^2/n.
}
Overall, we obtain from \eqref{Jac}
\eqal{
|x-y|^2\le C(n,K)ht\le r^2,
}
if $h\le C(n,K,t)r^2$.  As in case (i) above, we obtain the doubling inequality, noting the dependence $t=t(n,K,\|u\|_{C^{0,1}(B_1(0))})$.

\end{proof}

\section{Doubling for critical special Lagrangian equation by singular cutoff}


We establish the doubling inequality for critical phases $\Theta\ge(n-2)\pi/2$.  Recall \eqref{JacCrit}.

\begin{prop}
\label{prop:doubCrit}
    Let $u\in C^\infty$ solve \eqref{slag} on $B_2(0)$ with $\Theta\ge (n-2)\pi/2$.  Then for any $y\in B_{1/2}(0)$ and $r<1/4$,
    \eqal{
    \sup_{B_{1/4}(y)}a_1(D^2u)\le C(r,n,\|u\|_{C^{0,1}(B_1(0))})\sup_{B_r(y)}a_1(D^2u).
    }
    By the pinching \eqref{order}, we obtain properness, and conclude
    \eqal{
    \sup_{B_{1/4}(y)}|D^2u|\le C\Big(r,n,\|u\|_{C^{0,1}(B_1(0))},\sup_{B_r(y)}|D^2u|\Big).
    }
\end{prop}

In order to establish a doubling inequality, we are free to sacrifice all control inside a small ball.  Therefore, we can add a singularity to our cutoff inside this ball.

\begin{proof}

\medskip
\textbf{Step 1: cutoff.} Let $\al,h^{-1}\gg 1$.  We form the \textbf{singular cutoff} on $B_{1}(0)\setminus\{y\}$ of Korevaar exponential type:
\eqal{
\eta=\Big(e^{(S-\vp)/h}-1\Big)_+,
}
where, for $y\in B_{1/2}(0)$, we add an increasing concave term to Guan-Qiu's radial derivative:
\eqal{
\qquad \vp&=(x-y)\cdot Du-u+u(y)-\frac{\alpha^{-1}2^\al}{|x-y|^{2\al}},\\
S&=-1-\|(x-y)\cdot Du-u+u(y)\|_{L^\infty(B_{1/2}(y))}-\al^{-1}2^{3\al}.
}
Then $S-\vp< 0$ on $\pd B_{1/2}(y)$ and $S-\vp>0$ on $B_{1/4}(y)$ for $\al$ large enough.  In general, 
$$
B_{1/4}(y)\setminus\{y\}\subset\subset \text{supp}(\eta)\subset\subset B_{1/2}(y)\subset B_{1}(0).
$$
Note that we extend $\eta=0$ outside the connected component of $\{\eta>0\}$ in $B_1\setminus\{y\}$ containing the hole at $x=y$.

\smallskip
\textbf{Step 2: test function.} We next consider the maximum point $p$ of $\eta a_1$ on $B_{1/2}(y)\setminus B_r(y)$.  If $p$ is in the interior, then suppose that $\lambda_1=\cdots=\lambda_m>\lambda_{m+1}$ at $x=p$.  It follows that $a_m$ in Proposition \ref{prop:JacCrit} is smooth near $x=p$ and attains its maximum at $x=p$.  As in the Jacobi calculation \eqref{reverse_Jac}, we obtain at $p$
\eqal{
\label{max}
|\nabla_g\vp|^2\le h\Delta_g\vp.
}
We suppose $D^2u$ is diagonalized at $p$ with $\lambda_i=u_{ii}$ and $\lambda_1\ge\cdots\ge \lambda_n$.  We also denote $z_i=x_i-y_i$.  Then the increasing singular term increases the left hand side:
\eqal{
\label{lhs}
|\nabla_g\vp|^2=\sum_i z_i^2\frac{\Big(\lambda_i+Z^{-\al-1}\Big)^2}{1+\lambda_i^2},\qquad Z:=|z|^2/2.
}
The right hand side has an extra negative term from the concave singular cutoff:
\eqal{
\label{rhs}
\Delta_g\vp=\sum_i\frac{\lambda_i+Z^{-\al-1}}{1+\lambda_i^2}-(\al+1)Z^{-\al-2}\sum_i\frac{z_i^2}{1+\lambda_i^2}.
}
We emphasize that the correct signs in these equations require the extra term to be singular.  This is usually a fatal problem, but restricting to $|x-y|\ge r$, we encounter no issues.

\medskip
\noindent
\textbf{CASE 1: $|z_n|\ge |z|/\sqrt{n}$:} Using $|\lambda_n|\le |\lambda_i|$ in \eqref{order}, inequality \eqref{max} becomes
\eqal{
\label{max1}
Z\frac{\Big(\lambda_n+Z^{-\al-1}\Big)^2}{1+\lambda_n^2}\le C(n)h\Big(\frac{|\lambda_n|+Z^{-\al-1}}{1+\lambda_n^2}-\frac{c(n)(\al+1)Z^{-\al-1}}{1+\lambda_n^2}\Big).
}
\textbf{Hard subcase $|\lambda_n+Z^{-\al-1}|\le 4Z^{-\al-1}$:} this means $|\lambda_n|\le 5Z^{-\al-1}$.  Using the last negative term, we obtain
\eqal{
\al+1\le C(n).
}
This is a contradiction for fixed $\al=\al(n,\|u\|_{C^{0,1}(B_1(0))}$ large enough.  This case is hard because $h\ll 1$ is unavailable.

\smallskip
\noindent
\textbf{Easy subcase $|\lambda_n+Z^{-\al-1}|>4Z^{-\al-1}$:} this means $|\lambda_n|\ge 3Z^{-\al-1}$.  If $\lambda_n<0$,
\eqal{
C(n)h|\lambda_n|&\ge Z\Big(-\lambda_n-Z^{-\al-1}\Big)^2\\
&\ge cZ\lambda_n^2\\
&\ge cZ|\lambda_n|Z^{-\al-1}.
}
The $\lambda_n\ge 0$ case gives the same result.  In fact, in $B_{1/2}(y)\setminus B_{r}(y)$, we have $Z\le 1/8$, so we obtain
\eqal{
h\ge 1/C(n).
}
This is a contradiction for $h=h(n,\|u\|_{C^{0,1}(B_2)}$ small enough.

\medskip
\textbf{CASE 2: $|z_i|\ge |z|/\sqrt{n}$ for $i<n$. } Since $Z^{-\al-1}>1$ on $B_{1/2}(y)\setminus B_r(y)$, and $\lambda_i\ge|\lambda_n|\ge 0$ by \eqref{order}, the left hand side \eqref{lhs} becomes
\eqal{
|\nabla_g\vp|^2\ge c(n)Z\frac{\lambda_i^2+2\lambda_iZ^{-\al-1}+Z^{-2(\al+1)}}{1+\lambda_i^2}> c(n)Z\ge c(n)r^2.
}
Then \eqref{max} becomes
\eqal{
r^2/2\le C(n)h\Big(1+Z^{-(\al+1)}\Big)\le C(n)hr^{-2(\al+1)}.
}
This is a contradiction for $h(r,\al,n)=h(r,n,\|u\|_{C^{0,1}(B_2)})$ small enough.  

\medskip
\noindent
\textbf{Conclusion of Step 2:} the max must occur on the boundary.  Since $\eta=0$ on $\pd B_{1/2}(y)$,
\eqal{
\sup_{B_{1/2}(y)\setminus B_r(y)}\eta b= \sup_{\pd B_r(y)}\eta b\le C(r,n,\|u\|_{C^{0,1}(B_1)(0)})\sup_{B_r(y)}b.
}

\medskip
\noindent
\textbf{Step 3: doubling inequality.  }  For $r<1/4$, the above conclusion gives
\eqal{
\sup_{B_{1/4}(y)}b&\le \sup_{B_{1/4}\setminus B_r(y)}b+\sup_{B_r(y)}b\\
&\le C\sup_{B_{1/4}(y)\setminus B_r(y)}\eta b+\sup_{B_r(y)}b\\
&\le C\sup_{B_r(y)}b
}
Here, $C=C(r,n,\|u\|_{C^{0,1}(B_1(0))})$.  
\end{proof}

\section{Proof of the Theorems}

Let $u_k\in C^\infty$ solve \eqref{slag} on $B_2(0)$ with $\|u_k\|_{C^{0,1}(B_1(0))}\le A$ but blowup $|D^2u_k(0)|\to\infty$.  We choose a uniformly convergent subsequence in $B_1(0)$ to viscosity solution $u\in C^0(\ov{B_1(0)})$ of \eqref{slag}.  

\smallskip
\noindent
\textbf{Step 1: Partial regularity of the limit.  }There are two cases, and we claim Alexandrov is valid in both:

\smallskip
(i) Suppose $u_k$ are convex solutions or semi-convex with a Jacobi inequality.  It follows that $u$ is also convex, and Alexandrov's theorem shows that $D^2u$ exists a.e. in $B_1(0)$.

\smallskip
(ii) Suppose $\Theta\ge (n-2)\pi/2$.  Then $u_k$ is $n-1$ convex by \eqref{n-1}, so by Lemma \ref{lem:n-1}, $u$ is also $n-1$ convex in the viscosity sense.  Then Proposition \ref{prop:Alex} shows that Alexandrov's theorem is true for $n\ge 3$.  

\smallskip
Using Alexandrov, we choose $y\in B_{1/2}(0)$ such that $|y|\le R(n,K,A)/2$ is sufficiently close to $x=0$ as in Propositions \ref{prop:doubCon}, \ref{prop:doubSemi}, and \ref{prop:doubCrit}.  Letting $Q(x)$ be the Taylor polynomial of $u$ at $x=y$, we have $|u(x)-Q(x)|\le \sigma(|x-y|)$ for some modulus $\sigma(r)=o(r^2)/r^2$ as $r\to 0$.  This implies $Q$ solves \eqref{slag}, using quadratic comparison functions.

\smallskip
\noindent
\textbf{Step 2: Flattening the error.  }As in \cite[pg 17]{SY23a}.  We let error $v_k=u_k-Q$, then rescale
\eqal{
\bar v_k(\bar x)=r^{-2}v_k(y+r\bar x),\qquad \bar x\in B_1(0).
}
Then 
\eqal{
\label{close}
\|\bar v_k\|_{L^\infty(B_1(0))}&\le r^{-2}\|u_k(y+r\bar x)-u(y+r\bar x)\|_{L^\infty(B_1(0))}+\|\frac{u(y+r\bar x)-Q(y+r\bar x)}{r^2}\|_{L^\infty(B_1(0))}\\
&\le r^{-2}o(k)/k+\sigma(r).
}
The last inequality comes from uniform convergence and Alexandrov.  

\smallskip
\noindent
\textbf{Step 3: Savin stability of partial regularity.  }Since $Q$ is a solution of \eqref{slag}, observe that $\bar v_k$ solves the fully nonlinear elliptic PDE on $B_1(0)$
\eqal{
G(D^2\bar v_k)=\sum_{i=1}^n\Big[\arctan\lambda_i(D^2Q+D^2\bar v_k)-\arctan\lambda_i(D^2Q)\Big]=0.
}
We see that $G(0)=0$, and Savin's conditions are satisfied.  We use Proposition \ref{prop:Savin} to find $c_1$.  In \eqref{close}, we can choose $r=r(\sigma)\ll 1$, then all $k\ge k(r(\sigma))\gg 1$, such that $\|\bar v_k\|_{L^\infty(B_1(0))}\le c_1$.  By Savin Proposition \ref{prop:Savin}, we deduce that $\|\bar v_k\|_{C^{2,\alpha}(B_{1/2}(0))}\le 1$.  Equivalently, if we relabel $r/2$ as $r=r(\sigma)$,
\eqal{
\label{stable}
\|u_k\|_{C^{2,\alpha}(B_{r}(y)}\le C(n,Q,\sigma).
}

\smallskip
\noindent
\textbf{Step 4: Doubling to propagate the partial regularity.  }By Propositions \ref{prop:doubCon}, \ref{prop:doubSemi}, or \ref{prop:doubCrit}, we use \eqref{stable} to obtain for smooth solutions $u_k$
\eqal{
\sup_{B_{R}(y)}|D^2u_k|&\le C(r,n,K,A,a,\underline{\sup_{B_r(y)}|D^2u_k|})\\
&\le C(r,n,K,A,a,\underline{C(n,Q,\sigma)}).
}
Since $B_{R/2}(0)\subset B_R(y)$ and $r=r(\sigma)$, we obtain overall
\eqal{
|D^2u_k(0)|\le \sup_{B_{R/2}(y)}|D^2u|\le C(\sigma,n,K,A,Q,a).
}
This contradicts the blowup assumption.  We conclude the proof.

\section*{Acknowledgments}

I thank Ovidiu Savin for pointing out the reference \cite{CW93} and Yu Yuan for comments.



\end{document}